\documentclass[12pt]{amsart}
\usepackage{amsfonts}

\newtheorem{theorem}{Theorem}[section]
\newtheorem{corollary}[theorem]{Corollary}

\newtheorem{lemma}[theorem]{Lemma}
\newtheorem{proposition}[theorem]{Proposition}
\newtheorem{example}[theorem]{Example}

\newtheorem{remark}[theorem]{Remark}
\newtheorem{remarks}[theorem]{Remarks}
\newtheorem{definition}[theorem]{Definition}

\def\Cc{\hbox{\sf C\kern -.47em {\raise .48ex \hbox{$\scriptscriptstyle |$}}
   \kern-.5em {\raise .48ex \hbox{$\scriptscriptstyle |$}} }}

\newcommand{\la}{\lambda}
\newcommand{\si}{\sigma}

\def\card{{\rm card\,}}

\newcommand{\be}{\begin{equation}}
\newcommand{\ee}{\end{equation}}

\newcommand{\CC}{\mathbb{C}} 
\newcommand{\NN}{\mathbb{N}}
\newcommand{\RR}{\mathbb{R}}
\newcommand{\e}{\varepsilon}

\newcommand{\de}{\delta}

\def\la{\lambda}

\begin{document}

\baselineskip 5.8mm

\title[Spectral theory in max algebra]
{On some aspects of spectral theory for infinite bounded non-negative matrices in max algebra}

\author{Vladimir M\"uller, Aljo\v{s}a Peperko}

\begin{abstract}
\baselineskip 7mm

Several spectral radii formulas for infinite bounded nonnegative matrices in max algebra are obtained. We also prove some Perron-Frobenius type results for such matrices. In particular, we obtain results on block triangular forms, which are similar to results on Frobenius normal form of $n \times n$ matrices. Some continuity results are also established.
\end{abstract}

\maketitle

\noindent
{\it Math. Subj.  Classification (2010)}: 15A18, 15A80, 47J10, 15A60, 15B48, 47H07.
\\

\section{Introduction}
\vspace{1mm}

The algebraic system max algebra and its isomorphic versions (max-plus algebra, tropical algebra) provide 
an attractive way of describing a class of non-linear problems appearing for instance in manufacturing and transportation scheduling, information 
technology, discrete event-dynamic systems, combinatorial optimization, mathematical physics, DNA analysis, ...(see e.g. \cite{Bu10, 
BCOQ92, G92, Heidergott, B98, BSD95, GMW18, P11, Z} and the references cited there).
Max algebra's usefulness arises from a fact that 
these non-linear problems become linear when described in the max algebra language. Moreover, max algebra techniques were used to solve
certain linear algebra and graph theoretical problems (see e.g. \cite{ED08, MP12, GPZ}). In particular, tropical polynomial methods 
improved the accuracy of the numerical computation of the eigenvalues of a matrix polynomial (see e.g. \cite{ABG04, ABG05, GS08, AGS13,  BNS13} 
 and the references cited there).

The max algebra consists of the set of 
non-negative numbers with sum $a \oplus b = \max\{a,b\}$ and the standard product $ab$, where $a,b \ge 0$. 
A matrix $A=[A_{ij}]_{i,j=1} ^ n $ is non-negative, if  $A_{ij} \ge 0$ for all $i,j \in \{1,2,\ldots, n\}$. 
Let $\RR^{n \times n}$ ($\CC^{n \times n}$) be the set of all $n \times n$ real (complex) matrices and  
$\RR^{n \times n} _+$ the set of all $n \times n$ non-negative matrices. The entries of a matrix are also denoted by $a_{ij}$, $a_{i,j}$ or $A_{i,j}$.
The operations between matrices and vectors in  max algebra are defined by analogy with the usual linear algebra. 
The product of  non-negative matrices $A$ and $B$ 
in max algebra is denoted by $A \otimes B$, where $(A\otimes B)_{ij}=\max _{k=1, \ldots , n }A_{ik}B_{kj}$ and the sum $A\oplus B$ in  max algebra
is defined by $(A\oplus B)_{ij}= \max\{A_{ij}, B_{ij}\}$. The notation
$A^2 _{\otimes}$ means $A \otimes A$, and $A^k _{\otimes}$ denotes the $k$-th max power of $A$. If $x=(x_i) _{i= 1, \ldots , n}$ is a non-negative 
 vector, then the notation $A \otimes x$ means 
$(A \otimes x)_i =\max _{j =1, \ldots , n}A_{ij}x_j$. The usual
associative and distributive laws hold in this algebra. 

The role of the spectral radius of $A\in \RR^{n\times n} _+$ in max algebra is played by the maximum cycle geometric mean $r(A)$, which is defined by
\be
r(A)=\max\Bigl\{(A_{i_1 i_k}\cdots A_{i_3 i_2}A_{i_2 i_1})^{1/k}: k \in \mathbb{N} \;\;\mathrm{and}\;\; i_1,\dots,i_k\in\{1,\dots,n\}\Bigr\}
\label{cgm}
\ee
and equal to
$$r(A)=\max\Bigl\{(A_{i_1 i_k}\cdots A_{i_3 i_2}A_{i_2 i_1})^{1/k}: k \le n \;\;\mathrm{and}\;\; i_1,\dots,i_k\in\{1,\dots,n\} \;\;
\mathrm{mutually} \;\; \mathrm{distinct} 
\Bigr\}.$$

A digraph $\mathcal{G}(A)= (N(A),E(A))$ associated to $A\in \RR^{n\times n} _+$ is defined  by setting $N(A) =\{ 1,...,n\}$ and letting $( i,j) \in E(A)$ whenever $A_{ij} > 0$. 
When this digraph contains at least one cycle, one distinguishes critical cycles, where the maximum in (\ref{cgm}) is attained. A graph with just one node and no edges will be called trivial.
A bit unusually, but in consistency with \cite{Bu10, BGC-G09, KSS12}, a matrix   $A\in \RR^{n\times n} _+$ is called irreducible if  $\mathcal{G}(A)$  is trivial ($A$ is $1\times 1$ zero matrix) or strongly connected (for each $i,j \in N(A)$, $i\neq j$ there
is a path in $\mathcal{G}(A)$  that starts in $i$ and ends in $j$).


There are many different descriptions of the maximum cycle geometric mean $r (A)$ (see e.g. \cite{EJS88, Bu10, P08, P06, MP12} and the references cited there). 
It is known that $r(A)$ is the largest max eigenvalue of $A$, i.e., $r(A)$ is the largest
$\la\ge 0$  for which there 
exists $x\in\RR_+^n$, $x \neq 0$ with $A\otimes x=\la x$.

Moreover, if $A$ is irreducible, then $r(A)$ is the 
unique max eigenvalue and every max eigenvector is positive (see e.g. \cite[Theorem 2]{B98}, \cite{Bu10, BCOQ92, BSD93}).
Also, the max version of the Gelfand formula holds for any  $A\in \RR^{n\times n} _+$, i.e.,  
\be 
r(A)= \lim _{m \to \infty} \|A^m _{\otimes} \|^{1/m}
\label{ed}
\ee
for an arbitrary vector norm $\|\cdot \|$ on $\RR^{n \times n}$ (see e.g. \cite{P08} and the references cited there).




 An eigenproblem in max-algebra 
and its isomorphic versions (and an eigenproblem for more general maps) has already received a lot of attention (see e.g.  \cite{B98, Bu10, Heidergott, BGC-G09, AGW04, MN02, KM97, S07} and the
 references  cited there). The results can be applied in different contexts, for instance in optimal control problems (here the max eigenvectors correspond to stationary solutions of the dynamic programming equations and the max eigenvalues correspond to the maximal ergodic rewards per time unit), in the study of discrete event systems, in statistical mechanics, in the study of delay systems, ... (see e.g. \cite{AGW04, MN02, BGC-G09} and the references cited there).
 
 Also infinite dimensional extensions of  spectral theory in max algebra (and more general settings) have already received substantial attention (see e.g. \cite{MN02, MP17, MP18, AGW04, MN10, N81, N86, Appel, LN11, KM97, LM05, S07} and the references cited there). In this article we continue this investigation by focusing on infinite bounded nonnegative matrices and their spectral properties in max algebra. The article is organized in the following way.
 
 In Section 2 we recall some definitions and results from \cite{MP17, MP18}, which are relevant in the rest of the article. In Section 3 we prove several spectral radii formulas for  infinite bounded nonnegative matrices in max algebra and prove some Perron-Frobenius type results for such matrices. In Section 4 we prove results on block triangular forms which are similar to results on Frobenius normal form of $n\times n$ nonnegative matrices \cite{B98, Bu10, Heidergott, BGC-G09}. We conclude the article with some continuity results in Section 5.

\vspace{3mm}

\section{Preliminaries}


An infinite (entrywise) non-negative matrix  $A=(a_{ij})_{i,j=1}^\infty = (a_{i,j})_{i,j=1}^\infty$ is called bounded if 
$$
\|A\|=\sup\{a_{ij}: i,j\in\NN\} < \infty. 
$$
Let  $\RR_+^{\infty\times \infty}$ denote the set of all  infinite bounded non-negative matrices. 
For $A, B \in \RR_+^{\infty\times \infty}$ and $x \in l^{\infty} _+$ we denote by $\oplus$ and $\otimes$ the sum and the product in max algebra, respectively, i.e., for $i, j \in \NN$ let 
$$(A\oplus B)_{ij} = \max \{a_{ij}, b_{ij}\}, \;\;(A \otimes B)_{ij} = \sup _{k \in \NN} a_{ik}b_{kj}, \;\; (A \otimes x)_i = \sup_{j\in\NN} a_{ij} x_j.$$
Let $A^k _{\otimes}$ denote the $k$-th power in max algebra. Let us point out that $\otimes$ here does not denote the tensor product.

Let $\{e_1,e_2,\dots \}$ be the standard basis in $\ell^\infty_+$. Then  
$$\|A\|=\sup _{j\in \NN} \|A \otimes e_j\| = \sup _{\|x\|=1, x\in  l^{\infty} _+} \|A \otimes x\| = \sup_{ x\in  l^{\infty} _+ , x\neq 0}  \frac{ \|A \otimes x\| }{\|x\|}$$
and $\|A\otimes B\| \le \|A\|\cdot\|B\|$, where $\|x\|=\sup _{i \in \mathbb{N}}|x_i|$ for $x \in l^{\infty}$.


For $i_0, i_1,i_2,\dots,i_k\in\NN$ let 
$$
A(i_k,\dots,i_0)=\prod_{t=0}^{k-1} a_{i_{t+1} i_t}.
$$

It is easy to see that 
$$
\|A^k _{\otimes}\|=\sup\{A(i_k,\dots,i_0):i_0,\dots,i_k\in\NN\}
$$
and
$$
\|A^{k+j} _{\otimes}\|\le\|A^k _{\otimes}\|\cdot \|A^j _{\otimes}\|
$$
for all $k,j\in\NN$.
It is well known that this implies that the sequence $\|A^k _{\otimes}\|^{1/k}$ is convergent and its limit equals to the infimum. The limit is called the spectral radius in max-algebra (the Bonsall cone spectral radius of the map $g_A: x \mapsto A\otimes x$ on the cone $l^{\infty} _+ $) and denoted by $r(A)$. Observe that the map 
$g_A: l^{\infty} _+ \to l^{\infty} _+ $ is Lipschitz with the Lipschitz constant $\|A\|$. For some theory on Bonsall's cone spectral radius see e.g. \cite{MN02, MP17, MP18, MN10, N81}.

For $x\in\ell^\infty_+$ let $r_x(A)=\limsup_{k \to\infty}\|A _{\otimes} ^k \otimes x\|^{1/k}$ be the local spectral radius of $A$ at $x$ in max algebra. It is easy to see (and known) that $r(A) =r_y (A)$ where $y=(1,1,1, \cdots)$.
The approximate point spectrum $\sigma_{ap}(A)$ in max algebra is defined as the set of all $t\ge 0$ such that 
$$
\inf\{\|A \otimes x-tx\|:x\in l^{\infty} _+,\|x\|=1\}=0.
$$
The point spectrum $\sigma_{p}(A)$ in max algebra is defined as the set all $t\ge 0$ such that $A \otimes x=tx$ for some $x\in l^{\infty} _+$, $x\neq 0$. Clearly $\sigma_{p}(A) \subset \sigma_{ap}(A)$.

Let
$m(A)=\sup_j r_{e_j} (A)$ and so $m(A) \le r(A)$. Let $s(A)=\inf\{\|Ax\|:x\in  l^{\infty} _+, \|x\|=1\}$ be the minimum modulus of $A$ and let $ d(A)=\lim_{n\to\infty}s(A^n)^{1/n}$
be the lower spectral radius of $A$ (see \cite{MP18}).
The following result was proved in \cite[Corollaries 2 and 3]{MP17}, \cite[Proposition 3.1, Theorem 3.5 and Example 4.13]{MP18}.

\begin{theorem} Let $A$ be an infinite bounded non-negative matrix.
Then
 
(i) $[m(A), r(A)] \subset \si_{ap}(A) \subset [d(A), r(A)]$, 

(ii) $r_x (A) \in  \si_{ap}(A)$ for all  $x\in\ell^\infty_+$, $x\neq 0$, 

(iii) $d(A) = \min \{ t:  t \in \si_{ap}(A)\}$ and $r(A) = \max \{ t:  t \in \si_{ap}(A)\}$.
\end{theorem}

\begin{remark} {\rm 
(i) It is known that in general $m(A) \neq r(A)$ and  $ \si_{ap}(A)$ may not be convex (see \cite[Example 7]{MP17} and \cite[Example 3.2]{MP18}).

(ii) For an $n \times n$ nonnegative matrix $A$ it is known that 
$$\si_ {ap}(A)=\si_ {p}(A) =\{t: \hbox{ there exists } j\in\{1,\dots,n\}, t=r_{e_j}(A)\}$$
and also that the above does not hold for  $A \in \RR_+^{\infty\times \infty}$ (\cite[Remark 3]{MP17}).
}
\end{remark}

\bigskip

Denote further
\be
\label{mu}
\mu(A)=\sup\Bigl\{
\Bigl(A(i_1,i_k,\dots,i_2,i_1)^{1/k}: k\in\NN, i_1,\dots,i_k\in \NN\Bigr\}.
\ee
Clearly $\mu(A)\le r(A)$. 
Furthermore, one can assume that the vertices $i_1,\dots,i_k$ in the definition of $\mu(A)$ are mutually distinct.

Recall that for finite matrices $A\in \RR_+^{n\times n}$ we have $r(A)=\mu(A)$. Moreover, in this case
$$
\mu(A)=\max\Bigl\{
\Bigl(A(i_1,i_k,\dots,i_2,i_1)^{1/k}: k\le n, 1\le i_1,\dots,i_k\le n \hbox{ are mutually distinct}\Bigr\}.
$$

For infinite matrices the equality $\mu(A)= r(A)$ is no longer true in general.

\begin{example}{\rm
Let $A\in \RR_+^{\infty\times\infty}$ be defined by $a_{i, i+1}=1$ for all $i\in\NN$ and $a_{ij}=0$ otherwise (backward shift). It is easy to see that $\mu(A)=m(A)=0$ and $r(A)=1$. } 
\end{example}


The following example shows that the supremum in the definition of $\mu (A)$ may not be attained.

\begin{example} {\rm
Let $A\in \RR_+^{\infty\times\infty}$ be defined by $a_{ii}=\frac{i}{i+1}$ for all $i\in\NN$ and $a_{ij}=0$ otherwise. Then $\mu(A)=r(A)=1$ but the supremum in (\ref{mu}) is not attained.
}
\end{example}

\section{Spectral radii formulas  for infinite matrices in max algebra}

For $k\in\NN$ and $A\in \RR_+^{\infty\times\infty}$ write
$$
c_k(A)=\sup\Bigl\{A(i_k,\dots,i_0): i_0,\dots,i_k\in\NN\hbox{ mutually distinct}\Bigr\}
$$
and $r'  (A)$ denote the {\it upper simple path geometric mean radius}, i.e.,  
\begin{equation}
r'  (A)= \limsup_{k\to\infty}c_k (A)^{1/k}.
\label{pathradius}
\end{equation}

\begin{theorem}
\label{rmax1good}
For $A\in \RR_+^{\infty\times\infty}$ we have
\be
r(A)=\max\{\mu(A), r' (A)\}.
\label{rmax1}
\ee
\end{theorem}

\begin{proof}
Clearly $r(A)\ge\max\{\mu(A), r'(A)\}$. 

Suppose  that $r(A)>\mu (A)$.  
If $\mu(A)=0$ then $c_k(A)=\|A _{\otimes} ^k\|$ for all $k\in\NN$ and so the statement is trivial.

Suppose that $\mu(A)> 0$. Without loss of generality we may assume that $\mu(A)=1$ and
$r(A)>\mu(A)=1$.

Let $n_0\in\NN$ and $0<\e< r(A)-1$. Then there exists $n\ge n_0$ and $i_0,\dots,i_n\in\NN$
such that $A(i_n,\dots,i_1,i_0) >(r(A)-\e)^n>\|A\|^{n_0}$.
Omit all cycles in the path $i_0,i_1,\dots,i_n$. We obtain mutually distinct $j_0,\dots,j_k$ such that $A(j_k,\dots,j_1,j_0)\ge A(i_n,\dots,i_1,i_0)>(r(A)-\e)^n>\|A\|^{n_0}$.
Hence $n \ge k\ge n_0$ and 
$$
c_k(A)\ge (r(A)-\e)^n\ge (r(A)-\e)^k.
$$
Hence $\limsup_{k\to\infty} c_k(A)^{1/k}\ge r(A)-\e$. Since  $\e>0$ was arbitrary, we have $r'(A)=\limsup_{k\to\infty} c_k(A)^{1/k}\ge r(A)$. So $r(A)=\max\{\mu(A), r' (A)\}$.
\end{proof}

\vspace{5mm}
Let $A\in\RR_+^{\infty\times\infty}$, $A=(a_{ij})_{i,j=1}^\infty$. 
Let
$m_e(A)=\limsup_{j\to\infty} r_{e_j} (A)$.

For $n\in\NN$ let $P_n:\ell^\infty_+\to\ell_+^\infty$ be the canonical projection defined by
$P_n(x_1,x_2,\dots)=(\underbrace{0,\dots,0}_n,x_{n+1},\dots)$.

Let $r_{ess} (A)=\lim_{n\to\infty}r(P_nAP_n)= \inf_{n\in \NN }r(P_nAP_n).$ Observe that in this particular case the classical linear algebra product $P_n A P_n$  coincides with the max algebra product $P_n \otimes A \otimes P_n$.

We have
$$
r(A)=\lim_{k\to\infty}\sup\Bigl\{A(i_k,\dots,i_0)^{1/k}: i_0,\dots,i_k\in\NN\Bigr\},
$$
$$
r_{e_j}(A)=\limsup_{k\to\infty}\sup\Bigl\{A(i_k,\dots,i_1,j)^{1/k}: i_1,\dots,i_k\in\NN\Bigr\},
$$
$$
r_{ess}(A)=\lim_{n\to\infty}\lim_{k\to\infty}\sup\Bigl\{\{A(i_k,\dots,i_0)^{1/k}: i_0,\dots,i_k\ge n+1\Bigr\}.
$$

Clearly 
$$
m_e(A)\le m(A)\le r(A)
$$
and
$$
r_{ess}(A)\le r(A).
$$

Next we show that in general $m_e(A)\le r_{ess}(A)$ is not true.

\begin{example} {\rm \label{irexample}
Let

$$
A= \left[ \begin{array}{ccccc}
0&\frac{1}{2}&\frac{2}{3}&\frac{3}{4}& \cdots \\
\frac{1}{2}&0&0&0& \cdots \\
\frac{2}{3}&0&0&0& \cdots \\
\frac{3}{4}&0&0&0& \cdots \\
\vdots&\vdots&\vdots&\vdots &  \ddots
\end{array} \right].$$

Clearly $r_{ess}(A)=0$ since $P_1AP_1=0$. However, $r_{e_j}(A)=1$ for all $j \ge 2$. 
Indeed, $r(A) \le \|A\|=1$ and for $j\ge 2$ we have
$$
A(1,\underbrace{n,1,n\dots,n,1}_k,j)=\Bigl(\frac{n-1}{n}\Bigr)^{2k}\cdot\frac{j-1}{j}
$$
for all $k,n\in\NN$. So
$$
r_{e_j}(A)\ge\limsup_{k\to\infty}\|A_\otimes^{2k+1}e_j\|^{1/(2k+1)}\ge \frac{n-1}{n}.
$$
Since $n\in\NN$ was arbitrary, $r_{e_j}(A)=1$.
So $m_e(A)=m(A)=\mu (A)= r(A)=1$, while $r_{ess}(A)=0$. }
\end{example}

The following example shows that it may happen that $\mu (A) > m_e (A)$.

\begin{example}{\rm
Let $A=(a_{ij})_{i,j=1}^\infty$, where $a_{ii} =\frac{1}{i}$ for all $i \in \NN$ and $a_{ij} =0$ otherwise. Then $r_{e_j} (A) = \frac{1}{j}$ for all $j \in \NN$ and so $m_e (A)=0$. Also $r'(A)= r_{ess}(A) = 0$, but $\mu (A) = m(A) = r(A)=1$.}
\end{example}

\begin{theorem}
\label{rmax2good}
Let $A\in \RR_+^{\infty\times\infty}$. Then $\mu(A)\le m(A)$ and $r'(A)\le r_{ess}(A)$.
Consequently, 
\begin{equation}
r(A)=\max\{r_{ess}(A),m(A)\}.
\label{rmax2}
\end{equation}
\end{theorem}

\begin{proof}
Let $i_1,\dots,i_k\in\NN$. We have
$$
r_{e_{i_1}}(A)\ge
\limsup_{n\to\infty}\|A^{nk}e_{i_1}\|^{1/nk}\ge $$
$$
\limsup_{n\to\infty}\Bigl(A\bigl(i_1, \underbrace{i_k,\dots,i_1, \cdots ,i_k,\dots,i_1}_n\bigr)
\Bigr)^{1/nk}=
A(i_1,i_k,\dots,i_2,i_1)^{1/k}.
$$
Hence $m(A) \ge \mu (A)$.

\smallskip

To show that $r'(A) \le r_{ess} (A)$  we assume on the contrary that $r_{ess}(A)<r'(A)$. Without loss of generality we may assume that $\|A\|=1$. So there exists $k\in\NN$  such that $r(P_kAP_k)<r'(A)$.
Choose $0<\e<\frac{r'(A)-r(P_kAP_k)}{2}$.
Find $n_0\in\NN$ such that $\|(P_kAP_k)^n\|\le (r(P_kAP_k)+\e)^n$ for all $n\ge n_0$.
 
Find $i_0,i_1,\dots,i_N\in\NN$ mutually distinct  for a suitable sufficiently large $N\in\NN$ such that
$$
A(i_N,\dots,i_0)\ge (r'(A)-\e)^N
$$
(such $i_0,i_1,\dots,i_N$ and $N$ exist by (\ref{pathradius})).
Let 
$$
S=\{j: 0\le j\le N, i_j\le k\}.
$$
Clearly $\card S\le k$. We have
$$
A(i_N,\dots,i_0)=B\cdot C,
$$
where
$$
B=\prod\{a_{i_{j+1},i_j}:0\le j\le N, \{i_j,i_{j+1}\}\cap S\ne\emptyset\}\le\|A\|^{2k}=1
$$
and
$$
C=\prod\{a_{i_{j+1},i_j}:0\le j\le N, \{i_j,i_{j+1}\}\cap S=\emptyset\}.
$$
Then $C$ decomposes into at most $\card S+1\le k+1$ disjoint paths whose elements lie outside $\{1,\dots,k\}$.

If $j_0,j_1,\dots,j_m$ are mutually distinct elements outside $\{1,\dots,k\}$ then
$$
A(j_m,\dots,j_0)\le\|A\|^m=1\qquad(\hbox{if } m<n_0) \qquad\hbox{and }
$$
$$
A(j_m,\dots,j_0) \le (r(P_kAP_k)+\e)^m\quad(\hbox{if }m\ge n_0).
$$
Thus
$$
C\le (r(P_kAP_k)+\e)^{N-(k+1)n_0-2k}.
$$
Hence
$$
r'(A)-\e\le  (BC)^{1/N} \le
(r(P_kAP_k)+\e)^{1-N^{-1}(k+1)n_0-2N^{-1}k}\to r(P_kAP_k)+\e
$$
as $N\to\infty$.
Since $\e>0$ was arbitrary, we have $r'(A)\le r(P_kAP_k)$, a contradiction. 

{\rm  So $\max\{r_{ess}(A), m(A)\}\ge\max\{r'(A),\mu(A)\}=r(A)$ by Theorem \ref{rmax1good}. The reverse inequality is clear.}
\end{proof}

\begin{remark} {\rm \label{maxs}
 By Theorems \ref{rmax1good} and \ref{rmax2good} it follows that for $A\in \RR_+^{\infty\times\infty}$ we also have
\begin{equation}
r(A)=\max\{r' (A),m(A)\} = \max\{r_{ess} (A), \mu (A)\} .
\end{equation}
 }
\end{remark}

\bigskip
Suppose that $r(A)\ne 0$. For $j\in\NN$ write 
$$
c(e_j)=\sup\Bigl\{\frac{A(j,i_{k-1},\dots,i_1,j)}{r(A)^k}:k\in\NN,i_1,\dots,i_{k-1}\in\NN\Bigr\}
$$
(with no exponent $1/k$ here). 
\bigskip

\begin{lemma}
Let $A\in \RR_+^{\infty\times \infty}$, $r(A)\ne 0$, $r(P_1AP_1)<1$ and $c(e_1)<1$. Then $r(A)<1$.
\label{important}
\end{lemma}

\begin{proof}
Suppose on the contrary that $r(A)\ge 1$. 
Without loss of generality we may assume that $r(A)=1$. Let $b\in (0,1)$ satisfy $r(P_1AP_1)<b$ and $c(e_1)<b$.
Since $r(P_1AP_1)<b$, there exists $m_0\in\NN$ such that
$$
A(i_m,\dots,i_0)\le b^m\qquad(m\ge m_0; \;\; i_0,\dots, i_m\ge 2).
$$
We have $r_{ess}(A)\le r(P_1AP_1)<b$. So $\mu(A)=r(A)=1$ by Remark \ref{maxs}.

Let $k\ge m_0+2$ satisfy $\|A\|^2 b^{k/2-2}<1$ and choose mutually distinct $i_0,i_1,\dots,i_{k-1}\in\NN$
such that $A(i_0,i_{k-1},\dots,i_1,i_0)^{1/k}>b^{1/2}$.

If $1\notin\{i_0,\dots,i_{k-1}\}$ then $r(P_1AP_1)\ge A(i_0,i_{k-1},\dots,i_0)^{1/k}>b^{1/2}\ge b$, a contradiction.

Let $1\in\{i_0,\dots,i_{k-1}\}$. Without loss of generality we may assume that $i_0=1$. Then
$$
b^{k/2}<
A(i_0,i_{k-1},\dots,i_0)\le
\|A\|^2\cdot \|(P_1AP_1)^{k-2}\|\le
\|A\|^2\cdot b^{k-2}.
$$
So $1<b^{k/2-2}\|A\|^2$, a contradiction.


\end{proof}

Under the assumption $r_{ess}(A)<r(A)$ we prove additional results.

\begin{theorem}
Let $A\in \RR_+^{\infty\times\infty}$ and $r_{ess}(A)<r(A)$. Then there exists $i_0\in\NN$ with 
$r_{e_{i_0}}(A)=\mu (A) =r(A)$. In particular, 
$m(A)=r(A)= \mu (A)$.
\label{e_i attained}
\end{theorem}

\begin{proof}  Without loss of generality we may assume that $r(A)=1$.

 Since $r_{ess}(A)<1$, there exists $n\in\NN$ with $r(P_nAP_n)<1$. By Lemma \ref{important}, there exists $i_0\le n$ with $c(e_{i_0})=1$. Indeed, if $c(e_i)<1$ for all $1\le i\le n$, then Lemma \ref{important} gives inductively $r(P_{n-1}AP_{n-1})<1$, $r(P_{n-2}AP_{n-2})<1, \dots, r(P_1AP_1)<1$, $r(A)<1$, a contradiction.

Let $i_0\le n$ satisfy $c(e_{i_0})=1$. So for each $\e \in (0,1)$ there exist $k_{\e}\in\NN$ and $i_1,\dots,i_{k_\e-1}\in\NN$ with $A(i_0,i_{k_\e-1},\dots,i_1,i_0)>1-\e$. So
$$
\min \{r_{e_{i_0}}(A), \mu (A) \}\ge
(1-\e)^{1/k_\e}\ge 1-\e.
$$
Since $\e \in (0,1)$ was arbitrary, $r_{e_{i_0}}(A)=1=\mu (A)$. Hence $m(A)=\mu (A)= r(A)=1$.
\end{proof}

\begin{lemma}
If $A\in \RR_+^{\infty\times\infty}$ such that $r_{ess}(A)<r(A)=1$, then $\sup_n\|A_\otimes^n\|<\infty$.
\label{power_bound}
\end{lemma}

\begin{proof} 
Since $r_{ess}(A)<1$, there exist $n_0\in\NN$ and $m_0\in\NN$ such that
$$
A(i_m,\dots,i_0)<1\qquad(m\ge m_0; \;\;  i_0,\dots,i_m>n_0).
$$
In particular, 
$$
C:=\sup\{A(i_m,\dots,i_0): m\in\NN, i_1,\dots,i_{m-1}>n_0\}<\infty.
$$

We have 
$$
\sup_{k\in\NN}\|A_\otimes^k\|=\sup\{A(i_k,\dots,i_0): k\in\NN, i_0,\dots,i_k\in\NN\}.
$$
Since $\mu(A)=r(A)=1$, we can omit in the path $(i_k,\dots,i_0)$ all cycles and assume without loss of generality that the indices $i_0,\dots,i_k$ are mutually distinct.
Let $S=\{j: i_j\le n_0\}$. Clearly $\card S \le n_0$. So $S$ divides the path $i_0,\dots,i_k$ into at most $n_0+1$ subpaths with vertices outside the set $\{1,\dots, n_0\}$. 
So
$$
A(i_k,\dots,i_0)\le \|A\|^{2n_0} C^{n_0+1}
$$
and consequently $\sup_{k\in\NN}\|A_\otimes^k\|<\infty.$

\end{proof}

\begin{theorem}
Let $A\in \RR_+^{\infty\times\infty}$ and $r_{ess}(A)<r(A)$. Then $r(A) \in\sigma_p (A)$.
\label{eig}
\end{theorem}
\begin{proof}
Without loss of generality we may assume that $r(A)=1$. By Lemma \ref{important}, there exist $n\in \NN$ and $i_0\le n$ with $c(e_{i_0})=1$.
Set $x=\bigoplus_{j=0}^\infty A^j \otimes e_{i_0}$. By Lemma \ref{power_bound}, $x\in\ell^\infty_+$. We have
$$
A\otimes x=\bigoplus_{j=1}^\infty A^j \otimes e_{i_0}\le x.
$$
On the other hand, $x=(A\otimes x)\oplus e_{i_0}$. Since $c(e_{i_0})=1$, for each $\e>0$ there exist $k_\e\in\NN$ and $i_1,\dots,i_{k_\e-1}\in\NN$ with $A(i_0,i_{k_\e-1},\dots,i_1,i_0)>1-\e$. Hence
 $A\otimes x\ge A^{k_\e} \otimes e_{i_0}\ge (1-\e)e_{i_0}$. Since $\e>0$ was arbitrary, $A \otimes x\ge e_{i_0}$ and $A\otimes x=x$. Hence $r(A)\in\sigma_p(A)$.
\end{proof}
\begin{remark}{\rm There are several closely related results to Theorem \ref{eig} in the literature (\cite[Theorem 3.4]{MN02}, 
\cite[Theorem 4.4]{MN10} and \cite[Theorem 3.14]{MP17}; see also \cite[Conjecture 4.1]{MN10}). At the moment it is not clear if Theorem \ref{eig} is a special case of some of these results (in particular, it is not clear what is the relation between $r_{ess} (A)$ and the essential spectral radii studied there). In any case, our  proof of  Theorem \ref{eig} is more elementary than the proofs of
(\cite[Theorem 3.4]{MN02}, \cite[Theorem 4.4]{MN10} and \cite[Theorem 3.14]{MP17}).  
}
\end{remark}

The assumption  $r_{ess}(A)<r(A)$ is necessary for the conclusion of Theorem \ref{eig} as the following example shows.


\begin{example}{\rm
\label{forward}
 Let  $a_{i, i-1}=1$ for all $i\in \NN$, $i\ge 2$ and $a_{i,j}=0$ otherwise ($A$ is a forward shift). Then $r(A)=r_{ess} (A)=r'(A)=m(A)=m_e(A)=1$, $\mu (A)=0$ and 
$1$ is not in $\sigma_p (A)=\emptyset $.}
\end{example}

\bigskip

We conclude this section with some additional results on irreducible matrices.
The weighted directed graph $\mathcal{D}(A)$ associated with $A \in  \RR_+ ^{\infty\times\infty}$ has the vertex set $\NN$ and edges $(i,j)$ 
from a vertex $i$ to a vertex $j$ with weight $a_{ij}$ if and only if $a_{ij}>0$. A matrix $A \in  \RR_+ ^{\infty\times\infty}$ is called irreducible if and only if   $\mathcal{D}(A)$ strongly connected (for each $i,j \in \NN$, $i\neq j$, there exists  a path from $i$ to $j$ in 
$\mathcal{D}(A)$).  Equivalently,  $A \in  \RR_+ ^{\infty\times\infty}$ is irreducible if and only if for each $(i, j) \in \NN \times \NN$ there exists $k=k(i,j)$ such that $(A^k _{\otimes} )_{ij} >0$.
A matrix $A \in  \RR_+ ^{\infty\times\infty}$ is called reducible if it is not irreducible. Equivalently, 
$A \in  \RR_+ ^{\infty\times\infty}$ is reducible if and only if  there exists a non-empty set $M\subset \NN$, $M\neq \NN$, such that
$a_{ij}=0$ for all $(i,j) \in M \times (\NN\setminus M)$.

Obviously, $\mu (A)>0$ if $A$ is irreducible. We say that $x \in l^{\infty} _+ $ is strictly positive (and we denote $x>0$) if $x_i >0$ for all $i \in \mathbb{N}$. The following result generalizes a well known  finite dimensional result to the infinite dimensional case (see also \cite{AGW04}).

\begin{proposition} Let $A \in  \RR_+ ^{\infty\times\infty}$ be irreducible. If $\lambda \in \sigma _p (A)$ and $A\otimes x = \lambda x$, $x \in l^{\infty} _+$, $x\neq 0$, then 
$x >0$ and $\lambda \in [\mu (A), r(A)]$.
\label{irreducible}
\end{proposition}

\begin{proof} Clearly $\lambda = r_x (A) \le r(A)$. Choose $i$ such that $x_i >0$. Then for each $m \in \NN$ there exists $k=k(m,i)$ such that $(A^k _{\otimes})_{mi}>0$ and so
$$\lambda ^k x_m  = (A^k _{\otimes} \otimes x)_m\ge  (A^k _{\otimes})_{mi} x_i >0. $$
Thus $\lambda >0$ and $x_m >0$ and so $x>0$. 
Also for each $m, n \in \NN $ we have
$$\lambda ^n x_m  = (A^n _{\otimes} \otimes x)_m\ge  (A^n _{\otimes})_{mm} x_m$$
and so $\lambda \ge  (A^n _{\otimes})_{mm} ^{1/n}$, which implies $\lambda \ge \mu (A)$ by (\ref{potence2}). This completes the proof.
\end{proof}

\begin{example}
\label{exam}{\rm Let $0<\e<1$. Let $A=(a_{ij})_{i,j=1}^\infty \in  \RR_+ ^{\infty\times\infty}$ be defined by $a_{1,j}=\e^j$, $a_{j+1,j}=1\quad (j\in\NN)$ and $a_{i,j}=0$ otherwise. It is easy to see that $A$ is irreducible, $r(A)=1$ and $\mu (A)= \e\ne r(A)$.
}
\end{example}
 The following result can be considered as a max algebra version of  the classical Jentzsch-Perron theorem for (linear) kernel (integral) operators.




\begin{theorem} 
\label{maxJentzsch}
 Let $A\in  \RR_+ ^{\infty\times \infty}$ be irreducible and let $r_{ess}(A)< r(A)$. Then $ \sigma _p (A) = \{r(A)\}$ and each max-eigenvector of $A$ is strictly positive.
\end{theorem} 

\begin{proof}
By Theorem \ref{eig} 
we know that  
$ r(A) \in \sigma _p (A)$. By Remark \ref{maxs}, 
$r(A)= \mu (A)$ and so  $ \sigma _p (A) = \{r(A)\}$ and each max-eigenvector of $A$ is strictly positive by Proposition \ref{irreducible}.

\end{proof}
\begin{remark}{\rm The assumption $r_{ess}(A)< r(A)$ cannot be omitted in Theorem \ref{maxJentzsch}.  If  $A \in  \RR_+ ^{\infty\times\infty}$ is the matrix from Example \ref{exam}, then $A$ is irreducible, $r_{ess}(A)=r(A)=1$ and $1\notin\si_p(A)$.
}
\end{remark}
\begin{example}{\rm Let $A$ be the matrix from Example \ref{irexample}. Then each max-eigenvector of $A$ is of the form $x \in l^{\infty} _+$, $x_n= \frac{n-1}{n}x_1$ for all $n\ge 2$ and $x_1 >0$.
}
\end{example}

\section{Block triangular forms}

In this section we prove that under suitable conditions a matrix $A \in \RR_+^{\infty\times\infty}$  is permutationally equivalent to a matrix in a block triangular form (i.e., there exists an infinite permutation matrix $P$ such that $PAP^T=P\otimes A \otimes P^T$ is a matrix in a suitable block triangular form).

As in \cite{MP17, MP18} a subset $C$ of $l^{\infty} _+$ is called a cone (with vertex 0) if 
$tC \subset C$ for all $t \ge 0$, where $tC =\{tx : x \in C \} $. A cone $C\subset l^{\infty} _+$ is called a max-cone if for every pair  $x,y\in C$ also $x \oplus y \in C$. A cone 
$C$ is called invariant for $A$ if $A\otimes x \in C$ for all $x \in C$ (i.e., if $g_A (C) \subset C$). For a set 
$S \subset l^{\infty} _+$ we denote by $\mathrm{span}_{\oplus} S$ the max cone generated by $S$, i.e., 
$\mathrm{span}_{\oplus} S$ is the set of all $x\in  l^{\infty} _+$ for which there exist $k=k(x)\in \NN$, $s_1,\ldots , s_k \in S$ and $\lambda _1, \ldots, \lambda _k \ge 0$ such 
that $x= \lambda _1 s_1 \oplus \cdots  \oplus \lambda _k s_k$.

%


First we state a simple observation.

\begin{lemma}
Let $A=(a_{ij})_{i,j=1}^\infty\in\RR_+^{\infty\times\infty}$. Let $i,j\in\NN$ and $a_{ji}>0$. Then $r_{e_i}(A)\ge r_{e_j}(A)$. 

Consequently, $\mathrm{span}_{\oplus} \{e_k: r_{e_k}(A)\le a\}$ is a max cone invariant for $A$ for every $a\in\RR_+$.
\end{lemma}

\begin{lemma}
\label{form1}
Let $A\in \RR_+^{\infty\times\infty}$ satisfy $m_e(A)<m(A)$. Then there exists a finite nonempty set $F\subset\NN$ such that in the decomposition $\NN=F\cup(\NN\setminus F)$ the matrix $A$ is  permutationally equivalent to a matrix in the form 
$$
\left[ \begin{array}{cc}
A_{11}&0\\
A_{21}&A_{22}
\end{array} \right],
$$
where $m(A)=m(A_{11})= r_{e_j}(A)$ for all $j\in F$ and $m_e(A_{22})=m_e(A)$, $r_{ess}(A_{22})=r_{ess}(A)$, $m(A_{22})<m(A)$ and $r(A_{22})=\max\{m(A_{22}), r_{ess}(A)\}$.
\end{lemma}

\begin{proof}
Without loss of generality we may assume that $m(A)=1$.

Then there exists $i_0$ such that $r_{e_{i_0}}(A)=1$, since $m_e (A)< m(A)=1$. Let $F=\{j: r_{e_j}(A)=1\}$.
Since $m_e(A)<m(A)=1$, $F$ is a finite set. It is easy to see that $A$ has the required form in the decomposition $\NN=F\cup(\NN\setminus F)$.
\end{proof}

A better decomposition can be obtained if we assume also that $r_{ess}(A)<r(A)$.

\begin{lemma}
\label{form2}
Let $A\in \RR_+^{\infty\times\infty}$ satisfy $r_{ess}(A)<r(A)$ and $m_e(A)<m(A)$. Then there exists a finite nonempty set $F\subset\NN$ such that in the decomposition $\NN=F\cup(\NN\setminus F)$ the matrix $A$ is  permutationally equivalent to a matrix in the form 
$$
\left[ \begin{array}{cc}
A_{11}&0\\
A_{21}&A_{22}
\end{array} \right],
$$
where $r(A_{11})=r(A)=\mu (A_{11})= m(A)= m (A_{11})= r_{e_j}(A)$ for all $j\in F$. Moreover, $r(A)\in\si_p(A)$ and the supremum (maximum) in the definition of $\mu (A_{11})$ is attained.
\end{lemma}


\begin{proof}
Without loss of generality we may assume that $r(A)=1$.

Let $
\left[ \begin{array}{cc}
A_{11}&0\\
A_{21}&A_{22}
\end{array} \right]
$ be the decomposition obtained in Lemma \ref{form1}.
Let $\e>0$ satisfy $m(A_{22})+\e<m(A)=1$. We have $\mu(A)=m(A)=r(A)=1$ by Theorem \ref{rmax2good} and Remark \ref{maxs}, so there exists $k\in\NN$ and $i_0,\dots,i_{k-1}\in\NN$ such that
$$
A(i_0,i_{k-1},\dots,i_1,i_0)^{1/k}>1-\e.
$$
Clearly $r_{e_{i_j}}(A)=r_{e_{i_0}}(A)>1-\e>m(A_{22})$ for all $j=0,\dots,k-1$. So $i_0,\dots,i_{k-1}\in F$ and $\mu(A_{11})>1-\e$.
Since $\e>0$ was arbitrary and since $F$ is a finite set, 
we have $\mu(A_{11})=1$.

By Theorem \ref{eig}, $r(A)\in\si_p(A)$, which completes the proof.

\end{proof}

\begin{theorem}
\label{good_form}
Let $A\in\RR_+^{\infty\times\infty}$ satisfy $m_e(A)<m(A)$. Then there exists a sequence (finite or infinite) of finite nonempty disjoint sets $F_1,F_2,\dots\subset\NN$ and a sequence of numbers $(m_k)$ satisfying $m(A)=m_1 > m_2 >\dots$ such that in the decomposition $\NN=F_1\cup F_2\cup\cdots\cup(\NN\setminus \bigcup F_j)$
 the matrix $A$ is  permutationally equivalent to a matrix in the form
\begin{equation}\label{form}
\left[ \begin{array}{ccccc}
A_{11}&0&0&\cdots&0\\
*&A_{22}&0&\cdots&0\\
*&*&A_{33}&\cdots&0\\
\vdots&&&\ddots&\vdots\\
*&*&*&\cdots&A_{\infty,\infty}
\end{array} \right],
\end{equation}
where
$r_{e_j}(A)=m(A_{kk})=m_k$ for all $j\in F_k$. If the sequence $(m_k)$ is finite, then $m_e(A)=m(A_{\infty,\infty})$. If the sequence $(m_k)$ is infinite, then $m_e(A)=\lim_{k\to\infty}m_k$.

If, in addition, $r_{ess}(A)<r(A)$ then there exists a decomposition with the above properties such that
$$
r(A_{kk})= \mu (A_{kk})= m_k 
$$
for all $k$ that satisfy $m_k>r_{ess}(A)$. Moreover, for such $k$ the supremum (maximum) in the definition of $\mu (A_{kk})$ is attained.

\end{theorem}

\begin{proof}
The decomposition is obtained using Lemma \ref{form1}, inductively. 

Let $r(A)>r_{ess}(A)$, $m_k>r_{ess}(A)$ and let
$$
A'=
\left[ \begin{array}{ccccc}
A_{kk}&0&0&\cdots&0\\
*&A_{k+1,k+1}&0&\cdots&0\\
*&*&A_{k+2,k+2}&\cdots&0\\
\vdots&&&\ddots&\vdots\\
*&*&*&\cdots&A_{\infty,\infty}
\end{array} \right].
$$
Then $r(A')=\max\{m_k, r_{ess}(A)\}=m_k>r_{ess}(A)=r_{ess}(A')$ and by Theorem \ref{e_i attained} we have $r(A')=m(A')=m_k>m_e(A)=m_e(A')$.  So the statement follows from 
Lemma \ref{form2}.
\end{proof}

Let $A\in\RR_+^{\infty\times\infty}$ satisfy $m_e(A)<m(A)$. Without loss of generality (otherwise apply a suitable permutational equivalence) we assume that $A$ has the form (\ref{form}).
Each $A_{kk}$ (for $k <\infty $) can be transformed by simultaneous permutations of the rows and columns to a Frobenius normal form (FNF)  (see e.g. \cite{BR97}, \cite{BGC-G09}, \cite{Bu10}, \cite{KSS12}, 
\cite{BSST13} and the references cited there)
$$
 \left[ \begin{array}{ccccc}
   A_{l_k} ^{[k]}   &  0 &  0 &  \ldots &   0 \\
   *   &  A_{l_k-1} ^{[k]}  &  0   &  \ldots &  0 \\
   \vdots & \vdots &  \vdots &  \ddots  &  \vdots\\
   * &  *    & * & \ldots &   A_{1} ^{[k]}
\end{array} \right],
$$
where $A_{1} ^{[k]}, \ldots,  A_{l_k } ^{[k]} $ are irreducible 
 square submatrices of $A_{kk}$. This gives a (permutationally equivalent) form of a matrix $A$ denoted by
\be
 \left[ \begin{array}{cccccc}
   B_{1}   &  0 &  0 &    \ldots &   0  \\
   *   &  B_{2}     &  0 &  \ldots &  0  \\
 \vdots & \vdots &  B_3 &   \vdots  &  0  \\
   \vdots & \vdots &    \vdots  &  \ddots & 0\\
   * &  *    &  * & \ldots & A_{\infty,\infty}
\end{array} \right],
\label{fnf}
\ee
where all $B_k$ are finite dimensional irreducible matrices. 
In general, the diagonal blocks of the above form are determined uniquely (up to a  simultaneous permutation of their rows and columns), however their order is not determined uniquely.

Let $m_e(A)<m(A)$ and let  $A$ be a matrix in  the form (\ref{fnf}). Next we define the  reduced digraph $\mathcal{R}(A)= (N_{\mathcal{R}} (A), E_{\mathcal{R}} (A))$. Here the 
 matrices 
$B_1, B_2,  \ldots , A_{\infty,\infty} $ from (\ref{fnf}) correspond to the (possible infinite) set $N_{\mathcal{R}} (A)$ of sets of nodes $N_1, N_2, \ldots ,N _{\infty}$
of the strongly connected components of a digraph  
$\mathcal{G}(A)= (N(A), E(A))$.
Note that in (\ref{fnf}) an edge from a node of $N_{\mu}$ to a node of $N_{\nu}$ in $\mathcal{G}(A)$ may exist only if 
$\mu \ge \nu $.
The set $E_{\mathcal{R}} (A)$ equals
$$\{(\mu,\nu): \;\; \mathrm{there}\;\; \mathrm{exist}\;\; k \in N_{\mu} \;\;\mathrm{and}\;\; j\in N_{\nu}\;\;\mathrm{ such}\;\; \mathrm{that}\;\; a_{kj}>0 \}.$$ 

By a class of $A$ we mean a node $\mu$ (or also the corresponding set $N_{\mu}$) of the reduced graph $\mathcal{R}(A)$. 
Class $\mu$ accesses class $\nu$, denoted by $\mu \to \nu$, if 
$\mu =\nu $ or if there exists a $\mu - \nu$ path in $\mathcal{R}(A)$ (a path that starts in $\mu$ and ends in $\nu$). A node $j$ of $\mathcal{G}(A)$ is accessed by a class $\mu$, denoted by $\mu \to j$, if $j$ belongs to a class $\nu$ such that $ \mu \to \nu$.


The following result, that describes 
$r_{e_j} (A)$ via the access relation under the additional condition $r_{ess}(A)<r_{e_j}(A)$, follows  from  Theorem \ref{good_form}.

\begin{corollary}
 Let  $A\in \mathbb{R}_+ ^{\infty \times \infty }$  such that $m_e(A)<m(A)$ and $r_{ess}(A)<r(A)$ and let  $A$, $B_1, B_2, \ldots, A_{\infty,\infty}$ be from  (\ref{fnf}) and $j\in \NN$. If 
 $r_{e_j} (A) > r_{ess}(A) $, then
$$r_{e_j} (A)= \max \{r (B_{\mu}): \mu \to j\}.$$
\label{access}
\end{corollary}

\begin{remark} {\rm 
The 
 cycle time vector $\chi(A)$ of $A\in  \mathbb{R}_+ ^{\infty \times \infty }$ (see \cite{Gu94} for the $n\times n$ case) is a vector in $l^{\infty} _+$ with entries
$$[\chi(A)]_j = \limsup _{k \to \infty} (A^k _{\otimes}\otimes  y)_j ^{1/k}$$
where $y = 1$, the unit (column) vector. It is
not hard to check that $[\chi(A^T)]_j  = r_{e_j} (A)$, where $A^T$ denotes the transposed  matrix. Indeed, 
$\|A^k _{\otimes}\otimes e_j\| = y^T \otimes  A^k _{\otimes}\otimes e_j = e_j ^T \otimes  (A^ T)^ k _{\otimes}\otimes y $ and so
$$r_{e_j}(A)=\limsup_{k\to\infty} \|A^k _{\otimes}\otimes e_j\|^{1/k} =\limsup_{k\to\infty}( e_j ^T \otimes  (A^ T)^ k _{\otimes}\otimes y )^{1/k}= \limsup_{k\to\infty}( (A^ T)^ k _{\otimes}\otimes y )_j ^{1/k}=[\chi(A^T)]_j . $$ 
}
\label{Gunaw}
\end{remark}
\bigskip


\section{Continuity properties}

We consider the metric on $\RR_+^{\infty\times\infty}$ induced by $\|\cdot\|$, i.e., 
$$
d(A,B)= \|A-B\|=\sup\{|a_{ij}-b_{ij}|:i,j\in\NN\}.
$$


\begin{proposition}
\label{r_upperc}
The function $A\mapsto r(A)$ is upper semi-continuous on  $(\RR_+^{\infty\times\infty}, d).$
\end{proposition}

\begin{proof}
Let $A,B\in\RR_+^{\infty\times\infty}$ and $k\in\NN$. We have 
$$
(A_\otimes^k)_{j,i}=\sup\{A(i_k,i_{k-1},\dots,i_1,i_0): i_0=i, i_k=j\}
$$
and
$$
(B_\otimes^k)_{j,i}=\sup\{B(i_k,i_{k-1},\dots,i_1,i_0): i_0=i, i_k=j\}.
$$
Let $i_0=i, i_k=j$ and $i_1,\dots,i_{k-1}\in\NN$. Then
$$
\bigl|A(i_k,\dots,i_0)-B(i_k,\dots,i_0)\bigr|=
\bigl|a_{i_k,i_{k-1}}\cdots a_{i_1,i_0}-b_{i_k,i_{k-1}}\cdots b_{i_1,i_0}\bigr|
$$
$$
\le
\bigl|a_{i_k,i_{k-1}}\cdots a_{i_2,i_1}(a_{i_1,i_0}-b_{i_1,i_0})\bigr|+
\bigl|a_{i_k,i_{k-1}}\cdots a_{i_3,i_2}(a_{i_2,i_1}-b_{i_2,i_1})b_{i_1,i_0}\bigr|
$$
$$
\hskip1cm +
\cdots+\bigl|(a_{i_k,i_{k-1}}-b_{i_k,i_{k-1}})b_{i_{k-1},i_{k-2}}\cdots b_{i_1,i_0}\bigr|
\le
k\|A-B\|\max\{\|A\|^{k-1},\|B\|^{k-1}\}.
$$
So $\|A_\otimes^k-B_\otimes^k\|\le k\|A-B\|\max\{\|A\|^{k-1},\|B\|^{k-1}\}$ and the mapping $A\mapsto A_\otimes^k$ is continuous. So the function $A\mapsto \|A_\otimes^k\|^{1/k}$ is continuous and therefore the  function $A\mapsto r(A)=\inf_k \|A^k_\otimes\|^{1/k}$ is upper semicontinuous.

\end{proof}

In general the Bonsall cone spectral radius  is discountinuous (see also \cite{LN11}). This is shown by the following example, which is based on the classical example of Kakutani.

\begin{example}
\label{Kakutani}
{\rm
For $k\in\NN$, $k=2^j\cdot l$ with $l$ odd we write
$w_k=2^{-j}$.

Define $A\in\RR_+^{\infty\times\infty}$ by
$A_{i,i+1}=w_i$ and $A_{i,j}=0$ if $j\ne i+1$.

For $m\in\NN$ define $A_m\in\RR_+^{\infty\times\infty}$ by
$(A_m)_{i,j}=w_i$ if $j=i+1$ and $w_i\ge 2^{-m}$,
$(A_m)_{i,j}=0$ otherwise.

Clearly $\|A-A_m\|\to 0$. For each $m\in\NN$ we have $(A_m)_{\otimes} ^{2^{m+1}}=0$, and so $r(A_m)=0$ for all $m$.
Furthermore, 
$$
\|A ^{2^m} _{\otimes}\|=
\prod_{i=1}^{2^m}w_i=
1^{2^{m-1}}\cdot 2^{-2^{m-2}}\cdot 2^{-2\cdot2^{m-3}}\cdots 2^{-(m-1)}\cdot 2^{-2^m}.
$$
So
$$
\|A_{\otimes} ^{2^m}\|^{1/2^m}=\prod_{j=1}^{m-1} \Bigl(\frac{1}{2^j}\Bigr)^{2^{-j-1}}\cdot
\Bigl(\frac{1}{2^m}\Bigr)^{1/2^m}=
2^{-\sum_{j=1}^{m-1}j2^{-j-1}}\cdot\Bigl(\frac{1}{2^m}\Bigr)^{1/2^m} $$
$$\to 2^{-\sum_{j=1}^\infty j\cdot2^{-j-1}}=2^{-1}.
$$
Hence $r(A)=\lim_{m\to\infty}\|A_{\otimes} ^{2^m}\|^{1/2^m}=\frac{1}{2}\ne 0$.
}
\end{example}

\begin{remark} {\rm
Note that in the above example we have $A_1\le A_2\le\cdots$, so the spectral radius is discontinuous even for monotone sequences. So the infinite dimensional generalization to our setting of \cite[Proposition 3.7(ii)]{MP15} is not valid.
}
\end{remark}

The following results extends \cite[Proposition 3.7(i)]{MP15} to the infinite dimensional setting.

\begin{proposition}
\label{sap_upperc}
The function $A\mapsto\sigma_{ap}(A)$ is upper semi-continuous on  $(\RR_+^{\infty\times\infty}, d)$.
\end{proposition}

\begin{proof}
Let $t\ge 0$ and $t\notin\sigma_{ap}(A)$. So there exists $\de>0$ such that $\|A\otimes x-tx\|\ge\de$ for all $x\in l^{\infty} _+$, $\|x\|=1$.
If $\|B-A\|<\de/2$, then 
$$
\|B\otimes x-tx\|\ge\|A\otimes x-tx\|-\|A\otimes x-B \otimes x\|\ge\de/2
$$
for all unit vectors $x\in l^{\infty} _+$. So $t\notin\sigma_{ap}(B)$ and the mapping $B\mapsto\sigma_{ap}(B)$ is upper semicontinuous.
\end{proof}

\begin{remarks}{\rm (i) Propositions \ref{r_upperc} and \ref{sap_upperc} remain valid (with similar proofs) for Bonsall's cone spectral radius and approximate point spectrum of positively homogeneous bounded maps $A$ on a positive cone of a normed vector lattice. 
For neccesary definitions we refer the reader to e.g. \cite{MP17} or \cite{MP18}.

(ii) Example \ref{Kakutani} shows that in general the approximate point spectrum $\sigma_{ap}(\cdot)$ is not continuous. For a simpler example for finite matrices see e.g. also \cite[Example 3.6]{MP15}.
}
\end{remarks}

It is interesting that $\mu (\cdot)$ behaves in the opposite way than $r (\cdot)$.

\begin{proposition}
\label{mu_lower}  
The function $A\mapsto\mu(A)$ is lower semicontinuous  on  $(\RR_+^{\infty\times\infty}, d)$. 
\end{proposition}

\begin{proof}
Let $A,A_n\in \RR_+^{\infty\times\infty}$ such that $A_n\to A$.

If $\mu(A)=0$ then clearly $0=\mu(A)\le\liminf _{n \to \infty}\mu(A_n)$.

Let $\mu(A)>0$ and $\e\in (0,\mu(A))$. Find a cycle such that $A(i_1,i_k,\dots,i_2,i_1)\ge (\mu(A)-\e)^k$. Then
$$
\mu(A_n)\ge A_n(i_1,i_k,\dots,i_2,i_1)^{1/k}\to A(i_1,i_k,\dots,i_2,i_1)^{1/k}\ge \mu(A)-\e.
$$
So $\liminf _{n \to \infty}\mu(A_n)\ge\mu(A)$ and the function $A\mapsto\mu(A)$ is lower semi-continuous.
\end{proof}

The following example shows that  the function $A\mapsto\mu(A)$ is in general not continuous.

\begin{example} {\rm Let 
$A\in \RR_+^{\infty\times\infty}$  be defined by 
$A_{i,j}=\de_{i,j+1}$ (the Kronecker symbol), i.e., $A$ is the forward shift.
Let $B_k=A+E_k$, where $(E_k)_{1,k}=k^{-1}$ and $(E_k)_{i,j}=0$ otherwise.
Then $\mu(A)=0$, $B_k\to A$ and $\mu(B_k)=\frac{1}{k^{1/k}}\to 1$ as $k \to \infty$.
}
\end{example}

The following result follows from Propositions \ref{r_upperc} and \ref{mu_lower}.

\begin{corollary}
\label{cont}
Let $A\in \RR_+^{\infty\times\infty}$ satisfy $\mu(A)=r(A)$. Then the functions  $r(\cdot)$ and $\mu(\cdot)$ are continuous at $A$.
\end{corollary}

\begin{proof}
Let $A_n\to A$. We have
$$
r(A)\ge\limsup_{n\to\infty} r(A_n)
$$
by the upper semi-continuity of $r(\cdot)$.
Furthermore,
$$
r(A)=\mu(A)\le\liminf_{n\to\infty}\mu(A_n)\le\liminf_{n\to\infty} r(A_n)
$$
by the lower semi-continuity of the function $\mu(\cdot)$. Hence $r(A_n)\to r(A)$ whenever $A_n\to A$.

The continuity of  $\mu(\cdot)$ at $A$ is proved in a similar manner.
\end{proof}

By Corollary \ref{cont} and Theorem \ref{e_i attained} the following result follows.

\begin{corollary}
Let $A\in \RR_+^{\infty\times\infty}$ and $r_{ess}(A)<r(A)$. Then the functions  $r(\cdot)$ and $\mu(\cdot)$ are continuous at $A$.
\end{corollary}

\begin{definition}{\rm 
Let $(X, d)$ be a metric space. A mapping $f: X \to \RR$ is called H\"{o}lder continuous (of order $\alpha >0$) if there exists a constant $C\ge 0$ such that  the inequality 
\begin{equation}
\label{Holder}
|f(x)-f(y)| \le C d(x,y)^{\alpha}
\end{equation}
holds for all $x,y \in X$. The map $f$ is called locally H\"{o}lder continuous (of order $\alpha$) if for each $z\in X$ there exist $\varepsilon >0$ and  $C\ge 0$ (which may depend on $z$) such that (\ref{Holder}) holds for all $x,y \in B(z, \varepsilon)$, where $B(z, \varepsilon)$ denotes the closed ball in $X$ with the center $z$ and the radius $\varepsilon$. If $f$ is  locally H\"{o}lder continuous of order $\alpha =1$, then it is called locally Lipschitz continuous.

}
\end{definition}

\begin{remark}{\rm 
It was proved in the proof of Proposition  \ref{r_upperc} that for each $A, B \in  \RR_+^{\infty\times\infty}$ and $k\in \NN$ we have
$$
\|A_{\otimes} ^k -B_{\otimes} ^k\|\le k\|A-B\|\cdot\max\{\|A\|^{k-1},\|B\|^{k-1}\}.
$$
Thus for each $k\in \NN$ the map $A \mapsto A_ {\otimes} ^k$ is  locally Lipschitz continuous and thus also the map $A \mapsto \|A_{\otimes} ^k\|$ is  locally Lipschitz continuous,
since
$$| \|A_{\otimes} ^k\|-\|B_{\otimes} ^k\| | \le \|A_{\otimes}^k-B_{\otimes} ^k\|\le k\|A-B\|\cdot\max\{\|A\|^{k-1},\|B\|^{k-1}\}.$$
Thus the map  $A \mapsto \|A_{\otimes} ^k\|^{1/k}$ is locally H\"{o}lder continuous of order $\frac{1}{k}$.
}
\end{remark}
However, the following example shows that 
the mapping $r(\cdot)$ is in general not locally Lipschitz continuous on the set $\{A \in  \RR_+^{\infty\times\infty}: r(A)=\mu(A)\}$.

\begin{example}{\rm 
Let $A=0\in \RR_+^{\infty\times\infty}$. Then $r(A)=\mu(A)=0$. For $n\in\NN$ and $\e>\e'>0$ let $B_{n,\e}$ and $C_{n,\e,\e'}$ be given by
$$
(B_{n,\e})_{i,i+1}=\e\qquad(i< n)
$$
$$
(B_{n,\e})_{i,j}=0\qquad({\rm otherwise})
$$
$$
(C_{n,\e,\e'})_{i,i+1}=\e\qquad(i< n)
$$
$$
(C_{n,\e,\e'})_{n,1}=\e'
$$
$$
(C_{n,\e,\e'})_{i,j}=0\qquad({\rm otherwise}).
$$
Then $\|A-B_{n,\e}\|=\|A-C_{n,\e,\e'}\|=\e$ and
$\|B_{n,\e}-C_{n,\e,\e'}\|=\e'$ for all $n,\e,\e'$. Moreover, $r(B_{n,\e})=0$ and $
r(C_{n,\e,\e'})=(\e^{n-1}\e')^{1/n}\to\e$ as $n\to\infty$.
So for all $L>0$ and $\e>0$ there exist $B,C$ with $\|A-B\|\le\e$, $\|A-C\|\le \e$ and $|r(B)-r(C)|>L\|B-C\|$.
}
\end{example}

In contrast to the finite dimensional case (\cite[Proposition 5.2(ii)]{GMW18}), $r(\cdot)$ is  
in general not locally H\"{o}lder continuous of  any order $\alpha >0$ (and thus it is not locally Lipschitz continuous) even on the set $\{A \in  \RR_+^{\infty\times\infty}: r(A)=\mu(A)>0 \}$.
\begin{example}
{\rm Let $\alpha >0$. 
Set $n_1=1$. For each $k\ge 2$ find $n_k$ such that $$\displaystyle{(1+k^{-1})^{\frac {n_k-1}{n_k}}k^{\frac{-2}{\alpha n_k}}>1+\frac{1}{2k}}.$$
Let $X$ be a Banach lattice isomorphic to $\ell^\infty$ with the standard basis $e_{i,j}=\chi _{\{(i,j)\}}\quad(i\in\NN, 1\le j\le n_i$).
Define $A:X_+\to X_+$ by $A\otimes e_{1,1}=e_{1,1}$,
$$
A\otimes e_{i,j}=\Bigl(1+\frac{1}{i}\Bigr)e_{i,j+1}\qquad(i\ge 2, 1\le j<n_i)
$$
$$
A \otimes e_{i,n_i}=0.
$$
Then $r(A)=\mu(A)=1$.

For $k\ge 2$ define $B_k$ by
$$
B_k \otimes e_{1,1}=e_{1,1},
$$
$$
B_k \otimes e_{i,j}=\Bigl(1+\frac{1}{i}\Bigr)e_{i,j+1}\qquad(i\ge 2, 1\le j< n_i)
$$
$$
B_k \otimes e_{i,n_i}=k^{\frac{-2}{\alpha}}e_{i,1}\quad(i\ge 2).
$$
Then $\|A-B_k\|=k^{-2/\alpha}$ for all $k$.
Moreover,
$$
\lim_{k\to\infty}\frac{|r(B_k)-r(A)|}{\|B_k-A\|^\alpha}
=\lim_{k\to\infty} k^{2}\Bigl( (1+k^{-1})^{\frac{n_k-1}{n_k}} k^{-2/(\alpha n_k)}-1\Bigr)\ge\lim_{k\to\infty}k^{2}\cdot\frac{1}{2k}=\infty.
$$
So the function $r(\cdot)$ is not  locally H\"{o}lder continuous of order $\alpha$.
}
\end{example}

\begin{remark} {\rm
  The following weaker statement than local  H\"{o}lder continuity of  $A\mapsto \mu (A)$ on the set  $\{A\in  R_+^{\infty\times\infty} : \mu (A)  >0\}$ holds 
(and a related statement holds also for the map $A\mapsto r (A)$). \\
Let  $\mu (B)>0$. If $\mu (B) > \varepsilon >0$ and $\mu (A)>0$, then 
\be
\mu(B)-\e   \le  \mu (A) +  k ^{1/k}\|A-B\|^{1/k}\cdot\max\{\|A \|^{\frac{k-1}{k}},\|B\|^{\frac{k-1}{k}}\}
\label{weaker1}
\ee
for some $k \in \NN$.

Indeed,  there exists a cycle  such that $B(i_1,i_k,\dots,i_2,i_1)\ge (\mu(B)-\e)^k$. It follows from the proof of Proposition \ref{r_upperc}  that 
$$(\mu(B)-\e)^k \le B(i_1,i_k,\dots,i_2,i_1) \le A(i_1,i_k,\dots,i_2,i_1)+  k\|A-B\|\cdot\max\{\|A\|^{k-1},\|B\|^{k-1}\}$$
and so 
$$\mu(B)-\e  \le \left( A(i_1,i_k,\dots,i_2,i_1)+  k\|A-B\|\cdot\max\{\|A\|^{k-1},\|B\|^{k-1}\} \right) ^{1/k}$$
$$ \le A(i_1,i_k,\dots,i_2,i_1)^{1/k}+  k ^{1/k}\|A-B\|^{1/k}\cdot\max\{\|A\|^{\frac{k-1}{k}},\|B\|^{\frac{k-1}{k}}\} $$
$$ \le \mu (A) +  k ^{1/k}\|A-B\|^{1/k}\cdot\max\{\|A\|^{\frac{k-1}{k}},\|B\|^{\frac{k-1}{k}}\}. $$

{\rm Similarly, it can be proved that if  $\mu(A)>0$ 
and $\mu (B) > \varepsilon >0$, then 
\be
\mu(B)+\e   \ge   \mu(A) -  k ^{1/k}\|A-B\|^{1/k}\cdot\max\{\|A \|^{\frac{k-1}{k}},\|B\|^{\frac{k-1}{k}}\}
\label{weaker1}
\ee
for some $k \in \NN$.}
}
\end{remark}
\bigskip

\vspace{3mm}

\baselineskip 5mm

{\it Acknowledgments.}
The first author was supported by grants No. 20-22230L 
of GA CR and RVO:67985840. 
The second author acknowledges a partial support of the Slovenian Research Agency (grants P1-0222, J1-8133, J2-2512 and J1-8155). \\

\vspace{2mm}

\noindent
Vladimir M\"uller\\
Institute of Mathematics, Czech Academy of Sciences \\
\v{Z}itna 25 \\
115 67 Prague, Czech Republic\\
email: muller@math.cas.cz

\bigskip

\noindent
Aljo\v sa Peperko \\
Faculty of Mechanical Engineering \\
University of Ljubljana \\
A\v{s}ker\v{c}eva 6\\
SI-1000 Ljubljana, Slovenia\\
{\it and} \\
Institute of Mathematics, Physics and Mechanics \\
Jadranska 19 \\
SI-1000 Ljubljana, Slovenia \\
e-mail : aljosa.peperko@fs.uni-lj.si

\end{document}